\documentclass[12pt, a4paper]{article}
\usepackage{amsfonts}
\usepackage{mathrsfs}
\usepackage{latexsym}
\usepackage{xy}
\usepackage{amsfonts,amsmath,amssymb,amsthm}
\usepackage{color}

\xyoption{all}

\newcommand{\bcen}{\begin{center}}     \newcommand{\ecen}{\end{center}}
\newcommand{\bay}{\begin{array}}      \newcommand{\eay}{\end{array}}
\newcommand{\beq}{\begin{eqnarray*}}      \newcommand{\eeq}{\end{eqnarray*}}
\def\gl{\mathrm{gl.dim}}

\def\rad{\mathrm{rad}}

\def\max{\mathrm{sup}}

\def\Hom{\mathrm{Hom}}
\def\Ker{\mathrm{Ker}}
\def\Ext{\mathrm{Ext}}

\def\mod{\mathrm{mod}}
\def\Mod{\mathrm{Mod}}
\def\id{\mathrm{id}}
\def\pd{\mathrm{pd}}

\def\per{\mathrm{per}}
\def\proj{\mathrm{proj}}

\def\inj{\mathrm{inj}}

\def\RHom{\mathrm{RHom}}

\def\Tria{\mathrm{Tria}}
\def\Coloc{\mathrm{Coloc}}
\def\tria{\mathrm{tria}}
\def\thick{\mathrm{thick}}

\begin{document}

\newtheorem{theorem}{Theorem}[section]
\newtheorem{proposition}[theorem]{Proposition}
\newtheorem{lemma}[theorem]{Lemma}
\newtheorem{corollary}[theorem]{Corollary}
\newtheorem{remark}[theorem]{Remark}
\newtheorem{example}[theorem]{Example}
\newtheorem{definition}[theorem]{Definition}
\newtheorem{question}[theorem]{Question}
\numberwithin{equation}{section}

\title{\large\bf Reduction techniques of singular equivalences}

\author{\large Yongyun Qin}

\date{\footnotesize College of Mathematics and Statistics,
Qujing Normal University, \\ Qujing, Yunnan 655011, China. E-mail:
qinyongyun2006@126.com}

\maketitle

\begin{abstract} It is shown that a singular equivalence induced by tensoring
with a suitable complex of bimodules defines a singular equivalence of Morita
type with level, in the sense of Wang. This result is applied to
homological ideals and idempotents to produce new reduction techniques
for testing the properties of syzygy-finite and injectives generation
of finite dimensional algebras over a field.
\end{abstract}

\medskip


\medskip


\bigskip

\section{\large Introduction}

\indent\indent Throughout $k$ is a fixed field and all algebras are finite dimensional
associative $k$-algebras with identity, and all modules are finitely generated left modules
unless stated otherwise.
The {\it singularity category} $D_{sg}(A)$ of an algebra
$A$ is defined as the Verdier quotient of the bounded derived category of finitely generated modules over $A$ by the
full subcategory of perfect complexes \cite{Buch86}, and two algebras are called {\it singularly equivalent} if their singularity
categories are equivalent as triangulated categories.
In particular, derived equivalent algebras are singularly equivalent,
but the converse
is not true in general. For this reason, many scholars devote to extend the properties which are
preserved under derived equivalences to singular equivalences \cite{Chen18, CHQW18, Ska16, Wang19, ZZ13}.
In this respect, a special class of singular equivalences induced by bimodules
is crucially important. This was first studied by Chen and Sun \cite{CS12}
under the name of singular equivalence of Morita type,
and was generalized to
singular equivalence of Morita type with level in \cite{Wang15}.
This equivalence
captures rich structural information, and plays a central role
in the study of homological properties and singular equivalences \cite{Chen18, Ska16, Wang19, ZZ13}.
Therefore, it is of great interest to construct singular equivalences of Morita type with level,
and to find out which properties are invariant under these equivalences.
The purpose of this paper is to complement and extend some results in this literature.

In \cite{CLW20}, Chen-Liu-Wang gave a sufficient condition on when a tensor functor with a bimodule
defines a singular equivalence Morita type with level, and in \cite{Dal20}, Dalezios
proved that
for certain Gorenstein algebras, a singular equivalence induced from tensoring
with a complex of bimodules always induces a singular equivalence of Morita
type with level. Our first theorem is a complex version of
Chen-Liu-Wang's work, and it
generalizes the result of Dalezios
to arbitrary algebra (not limited to Gorenstein algebra).

\medskip

{\bf Theorem I.}  (Theorem~\ref{theorem-sing}) {\it Let $A$ and $B$ be finite-dimensional $k$-algebras
such that $A/\rad(A)$ and $B/\rad(B)$ are separable over $k$.
Consider a complex $X$ of finitely
generated $A$-$B$-bimodules which is perfect over $A$ and $B$.
Assume that $\RHom _A(X, A)$ is a perfect complex of left $B$-module, and
that $X \otimes _B^L- : D_{sg}(B)\rightarrow D_{sg}(A)$ is an equivalence.
Then there is an $A$-$B$-bimodule $M$ and a $B$-$A$-bimodule $N$
such that $(M,N)$ defines a singular equivalence of Morita type with level.
}

\medskip

Theorem I can be applied to homological ideals and idempotents to produce
singular equivalences of Morita
type with level.
Let $A$ be an algebra and let $J\subseteq A$ be a two-sided ideal. Following \cite{PX06},
$J$ is a {\it homological ideal} if the canonical map $A \rightarrow  A/J$ is a homological
epimorphism, that is, the naturally induced functor $\mathcal{D}^b(\mod A/J) \rightarrow
\mathcal{D}^b(\mod A)$ is fully faithful. In \cite{Chen14}, Chen proved that
if $J$ is a homological ideal which has finite projective dimension as an $A$-$A$-bimodule,
then there is a singular equivalence between $A$ and $A/J$.
In this paper we show that this equivalence is a singular equivalence
of Morita type with level (cf. Theorem~\ref{thm-homo-epic}).

Let $e \in A$ be an idempotent.
Then the functor $eA\otimes _A-: \mod A\rightarrow
\mod eAe$ induces a singular equivalence between
$A$ and $eAe$ if and only if $\pd _A(\frac {A/AeA}{\rad (A/AeA)})<\infty$ and $\pd _{eAe}eA<\infty$,
see \cite{Chen09, PSS14}. Similarly, $Ae\otimes _{eAe}-: \mod eAe\rightarrow
\mod A$ induces a singular equivalence if and only if $\id _A(\frac {A/AeA}{\rad (A/AeA)})
\linebreak
<\infty$ and $\pd Ae_{eAe}<\infty$,
see \cite{Shen20}. Applying Theorem I, we show that
$A$ and $eAe$ are
singularly equivalent of Morita type with level in these two cases, see Theorem~\ref{thm-idem-pd} and Theorem~\ref{thm-idem-id}.

Next, we turn to the question that which properties can be preserved under singular equivalence of Morita type with level.
It is known that the finitistic dimension conjecture is invariant under this equivalence \cite{Wang15}.
In this paper, we focus on
the properties of syzygy-finite, Igusa-Todorov, injectives generation
and projectives cogeneration, all of which are closely related to finitistic dimension conjecture
\cite{Ric19, Wei09, ZH95}.
We show that these properties are also invariant under singular equivalence of Morita type with level
(cf. Proposition~\ref{prop-syz} and Proposition~\ref{prop-inj-gen}).
As applications,
we obtain the following reduction techniques for testing these properties for
finite dimensional algebras over a field.

\medskip

{\bf Corollary I.} (Corollary~\ref{coro-homo-epic-syz})
{\it Let $A$ be a finite-dimensional algebra over a field and let $J\subseteq A$ be a homological
ideal which has finite projective dimension as an $A$-$A$-bimodule. Then
$A$ has the property of syzygy-finite (resp. Igusa-Todorov, injectives
generation, projectives cogeneration) if and only if so does $A/J$.
}

\medskip

\medskip

{\bf Corollary II.} (Corollary~\ref{cor-iedm-inj-gen})
{\it Let $A$ be a finite-dimensional algebra over a field $k$ with separable semisimple quotient, and
let $e \in A$ be an idempotent such that
$Ae\otimes _{eAe}^LeA$ is bounded in
cohomology. If $\pd _A(\frac {A/AeA}{\rad (A/AeA)})
<\infty$ or $\id _A(\frac {A/AeA}{\rad (A/AeA)})
<\infty$, then $A$ has the property of syzygy-finite (resp. Igusa-Todorov, injectives
generation, projectives cogeneration) if and only if so does $eAe$.
}

\medskip

Corollary II may be compared with a recent result by Cummings \cite{Cum19}.
For a ring $A$, Cummings proved that if $Ae\otimes _{eAe}^LeA$ is bounded in
cohomology, then (i) if $\id _A(\frac {A/AeA}{\rad (A/AeA)})
<\infty$ and injectives generate for $eAe$, then injectives generate for $A$;
(ii) if $\pd _A(\frac {A/AeA}{\rad (A/AeA)})
<\infty$ and projectives generate for $eAe$, then projectives generate for $A$.
Therefore, if we only consider finite-dimensional $k$-algebras
with separable semisimple quotients (for instance,
it is the case when $k$ is algebraically closed),
then the property of injectives
generation (resp. projectives cogeneration) between $A$ and $eAe$
can be
displayed more completely.

In Corollary II, the transition from $A$ to $eAe$ is called {\it vertex removal}
in some literature \cite{FS92, GPS18}. If we restrict our discussion to quiver algebras,
then we get a practical method for testing the properties of syzygy-finite,
Igusa-Todorov, injectives generation and projectives cogeneration
--- just removing the vertices where no relations start or no relations end (cf. Corollary~\ref{cor-path-alg}).

The paper is organized as follows. In section 2, we will recall some relevant
definitions and conventions.
In section 3 we prove Theorem I, and we show that a certain homological ideal
induces a singular equivalence of Morita type with level.
In section 4 we construct singular equivalences of Morita type with level by idempotents.
In section 5, we investigate the invariance of syzygy-finite and injectives generation
under singular equivalence of Morita type with level, and we prove Corollary I
and Corollary II. In particular, we give two examples to illustrate how our reduction techniques
can be used.

 \section{\large Definitions and conventions}\label{Section-definitions and conventions}

\indent\indent
Let $\mathcal{C}$ be a triangulated category which has all (set-indexed)
products and coproducts.
An object $X$ of $\mathcal{C}$ is {\it compact} if the functor $\Hom_\mathcal{C}(X, -)$
preserves coproducts. For a set $\mathcal{S}$ of objects
of $\mathcal{C}$, we denote by $\tria \mathcal{S}$ the smallest triangulated
subcategory of $\mathcal{C}$ containing $\mathcal{S}$,
and by $\thick \mathcal{S}$ (resp. $\Tria \mathcal{S}$, $\Coloc \mathcal{S}$) the smallest triangulated
subcategory of $\mathcal{C}$ containing $\mathcal{S}$
and closed under taking direct summands (resp. coproducts, products).
$\mathcal{S}$ is called a set of {\it compact generators} of $\mathcal{C}$
if all objects in $\mathcal{S}$ are compact and $\mathcal{C}=\Tria \mathcal{S}$.
In this paper, all functors between triangulated categories are assumed to be triangle functors.

\begin{definition}{\rm (\cite{BBD82})
Let $\mathcal{T}_1$, $\mathcal{T}$ and $\mathcal{T}_2$ be
triangulated categories. A {\it recollement} of $\mathcal{T}$
relative to $\mathcal{T}_1$ and $\mathcal{T}_2$ is given by
$$\xymatrix@!=4pc{ \mathcal{T}_1 \ar[r]^{i_*=i_!} & \mathcal{T} \ar@<-3ex>[l]_{i^*}
\ar@<+3ex>[l]_{i^!} \ar[r]^{j^!=j^*} & \mathcal{T}_2
\ar@<-3ex>[l]_{j_!} \ar@<+3ex>[l]_{j_*}}\ \ \ \ \ \ \ (R)$$
such that

(R1) $(i^*,i_*), (i_*,i^!), (j_!,j^*)$ and $(j^*,j_*)$ are adjoint
pairs;

(R2) $i_*$, $j_!$ and $j_*$ are full embeddings;

(R3) $j^*i_*=0$ (and thus also $i^!j_*=0$ and $i^*j_!=0$);

(R4) for each $X \in \mathcal {T}$, there are triangles

$$\begin{array}{l} j_!j^*X \rightarrow X  \rightarrow i_*i^*X  \rightarrow
\\ i_!i^!X \rightarrow X  \rightarrow j_*j^*X  \rightarrow
\end{array}$$ where the arrows to and from $X$ are the counits and the
units of the adjoint pairs respectively. }
\end{definition}

\begin{definition}{\rm (\cite{BGS88, QH16})
Let $\mathcal{T}_1$, $\mathcal{T}$ and $\mathcal{T}_2$ be
triangulated categories, and $n$ a positive integer. An {\it
$n$-recollement} of $\mathcal{T}$ relative to $\mathcal{T}_1$ and
$\mathcal{T}_2$ is given by $n+2$ layers of triangle functors
$$\xymatrix@!=4pc{ \mathcal{T}_1 \ar@<+1ex>[r] \ar@<-3ex>[r]_\vdots & \mathcal{T}
\ar@<+1ex>[r]\ar@<-3ex>[r]_\vdots \ar@<-3ex>[l] \ar@<+1ex>[l] &
\mathcal{T}_2 \ar@<-3ex>[l] \ar@<+1ex>[l]}$$ such that every
consecutive three layers form a recollement.}
\end{definition}

Let $A$ be a finite dimensional associative
algebra over a field $k$.
Denote by $\rad(A)$ the Jacobson radical of $A$. The semisimple quotient $A/ \rad(A)$
is called {\it separable} if $A/ \rad(A)$ remains semisimple under any extension of scalars to a field $K$
containing $k$. In particular, $A/ \rad(A)$ is separable if $k$ is an algebraically closed field.

Denote by $\Mod A$ the
category of left $A$-modules, and by $\mod A$, $\proj A$
and $\inj A$ the full subcategories consisting of all finitely
generated modules, finitely
generated projective modules and finitely
generated injective modules, respectively. We denote by
$\underline{\mod }A$ (resp. $\overline{\mod} A$) the projective
(resp. injective) stable category
of $\mod A$ modulo morphisms factoring
through projective (resp. injective) modules.

Let $\mathcal{X}$ be a subcategory of $\Mod A$. A (chain) complex $X$ over $\mathcal{X}$
is a set $\{ X_i\in \mathcal{X}, i \in \mathbb{Z}\}$
equipped with a set of homomorphisms $\{ d^i_X: X_i \rightarrow X_{i-1}, i \in \mathbb{Z}|d^i_X
d^{i+1}_X=0\}$. We usually write
$X=\{ X_i, d^i_X\}$. A chain map $f$ between complexes, say from $\{ X_i, d^i_X\}$
to $\{ Y_i, d^i_Y\}$
is a set of maps $f=\{f_i : X_i \rightarrow Y_i\}$ such that $f_{i-1}d^i_X=d^i_Yf_i$.
A complex $X=\{ X_i, d^i_X\}$ is
right (resp. left) bounded if $X_i=0$ for all but finitely many negative (resp. positive) integers
$i$. A complex $X$ is bounded if it is both left and right bounded, equivalently, $X_i=0$ for all but
finitely many $i$. We denote by $[1]$ the left shift functor on complexes.

Let $\mathcal{D}(\Mod A)$ (resp. $\mathcal{D}^b(\mod A)$)
be the derived category (resp. bounded derived category) of complexes over $\Mod A$ (resp. $\mod A$). Let
$K^b (\proj A)$ (resp. $K^b (\inj A)$) be the bounded homotopy category of complexes
over $\proj A$ (resp. $\inj A$).
Up to isomorphism, the objects in $K^{b}(\proj A)$ are
precisely all the compact objects in $\mathcal{D}(\Mod A)$. For
convenience, we do not distinguish $K^{b}(\proj A)$ from the {\it
perfect derived category} $\mathcal{D}_{\per}(A)$ of $A$, i.e., the
full triangulated subcategory of $\mathcal{D}(\Mod A)$ consisting of all
compact objects, which will not cause any confusion. Moreover, we
also do not distinguish $K^b(\inj A)$ (resp.
$\mathcal{D}^b(\mod A)$) from their essential images under the
canonical full embeddings into $\mathcal{D}(\Mod A)$.
Usually, we
just write $\mathcal{D} A$ (resp. $\mathcal{D}^b(A)$) instead of $ \mathcal{D}(\Mod A)$ (resp. $ \mathcal{D}^b(\mod A)$).

Let $A$ and $B$ be finite dimensional
algebras over a field $k$
and $F : \mathcal{D} A\rightarrow \mathcal{D} B$ be a triangle functor. We
say that $F$ {\it restricts} to $K^b (\proj )$ (resp.
$\mathcal{D}^b(\mod)$, $K^b(\inj )$) if $F$ sends $K^b (\proj A)$ (resp.
$\mathcal{D}^b(\mod A)$, $K^b(\inj A)$) to $K^b (\proj B)$ (resp.
$\mathcal{D}^b(\mod B)$, $K^b(\inj B)$).

Following \cite{Buch86, Orl04}, the {\it singularity category} of $A$ is defined to be the
Verdier quotient $D_{sg}(A) = \mathcal{D}^b(\mod A)/K^{b}(\proj A)$.
Let $A^e = A\otimes _k A^{op}$ be the enveloping algebra of $A$. We identify
$A$-$A$-bimodules with left $A^e $-modules. Denote by $\Omega _{A^e}(-)$ the syzygy functor on the
stable category $\underline{\mod} A^e$ of $A$-$A$-bimodules. The following terminology is
due to Wang \cite{Wang15}.

\begin{definition}
{\rm Let $_AM_B$ and $_BN_A$ be an $A$-$B$-bimodule and a $B$-$A$-bimodule,
respectively, and let $n \geq 0$. We say $(M,N)$ defines a {\it singular equivalence
of Morita type with level} $n$, provided that the following conditions are satisfied:

(1) The four one-sided modules $_AM$, $M_B$, $_BN$ and $N_A$ are all finitely generated
projective.

(2) There are isomorphisms $M\otimes _B N \cong \Omega _{A^e}^n(A)$ and
$N\otimes _A M\cong \Omega _{B^e}^n(B)$ in
$\underline{\mod} A^e$ and $\underline{\mod} B^e$, respectively.}

\end{definition}

\begin{remark}
{\rm If $(M,N)$ defines a singular equivalence
of Morita type with level $n$, then the functor
$M\otimes _B-$ induces a singular equivalence
between $A$ and $B$, that is, $M\otimes _B- : D_{sg}(B)\rightarrow D_{sg}(A) $
is a triangle equivalence. This equivalence
preserves many homological properties and homological conjectures, such as
Hochschild homology \cite{Wang15},
Fg condition \cite{Ska16}, Keller's conjecture \cite{CLhW20} and the finitistic dimension
conjecture \cite{Wang15}.}
\end{remark}

Recall that an algebra $A$ is {\it syzygy-finite} provided that there is an integer $s$ such that the class
of all $n$-th syzygies, where $n > s$, is representation finite, or equivalently, the number of nonisomorphic
indecomposable modules in the class is finite.
Such class of algebras
include algebras of finite global dimension, algebras of finite representation type, monomial
algebras and serial algebras.

From \cite{Wei09}, an algebra $A$ is called {\it Igusa-Todorov} if there are a fixed $A$-module $V$ and an integer $n$ such that every
$n$-th syzygy module $M$ fits into an exact sequence $0 \rightarrow
V_1 \rightarrow V_0 \rightarrow M \rightarrow 0$, where $V_1$, $V_0$
are some direct summands of finite direct sums of $V$.
Examples of such
algebras include syzygy-finite algebras, algebras with representation dimension
not more than three and algebras with infinite-layer length not more than three \cite{HLM09}.
It is known that
Igusa-Todorov algebras satisfy the finitistic dimension conjecture,
and the invariance of syzygy-finite and Igusa-Todorov
under recollements and derived equivalences is discussed in \cite{Wei17, WW20}.

Let $A$ be an algebra and $D:=\Hom _k(-,k)$ be the standard duality.
If $\Tria DA = \mathcal{D} A$ then we say that {\it injectives generate} for $A$,
and dually, if $\Coloc A = \mathcal{D} A$ then we say {\it projectives cogenerate} for $A$.
These concepts were proposed by Keller \cite{Kel01} as they are well-connected with
some homological conjectures. In particular, if injectives
generate for an algebra $A$, then $A$ satisfies the Nunke condition,
the Generalised Nakayama conjecture and the finitistic dimension conjecture \cite{Kel01, Ric19}.
Moreover, if projectives
cogenerate for $A$, then its opposite algebra $A^{op}$ satisfies
the finitistic dimension conjecture \cite{Ric19}.

Nowadays, there is no known example of
a finite dimensional algebra over a field for which injectives do not generate,
and the property of injectives generation has been verified
for commutative
algebras, Gorenstein algebras and monomial algebras \cite[Theorem 8.1]{Ric19}.
On the other hand, the properties of injectives generation and projectives
cogeneration are shown invariant under recollements and derived equivalences
of algebras \cite{Cum19, Ric19}.

\section{Singular equivalences induced by complexes}
\indent\indent In this section, we will investigate when a tensor functor giving by a bi-module
complex induces a singular equivalence of Morita type with level.
Let us make some notations. We denote by $\Omega _{A}(-)$ (resp. $\Omega _{A-B}(-)$) the syzygy functor on
the stable category of $A$-modules (resp. $A$-$B$-bimodules), and $\Omega _{\mathcal{D}^b(A)}(-)$ the syzygy functor on
derived category, up to some direct summands of projective
modules.
We point that $\Omega _A(M)= \Omega _{\mathcal{D}^b(A)}(M)$ for any $M\in \mod A$,
and we refer to \cite{AI07, Wei17} for more details on syzygies of complexes.

\begin{theorem}\label{theorem-sing}
Suppose that both $A/\rad(A)$ and $B/\rad(B)$ are separable over $k$.
Consider a complex $X$ of finitely
generated $A$-$B$-bimodules which is perfect over $A$ and $B$.
Assume that $\RHom _A(X, A)$ is a perfect complex of left $B$-module, and
that $X \otimes _B^L- : D_{sg}(B)\rightarrow D_{sg}(A)$ is an equivalence.
Then there is an $A$-$B$-bimodule $M$ and a $B$-$A$-bimodule $N$
such that $(M,N)$ defines a singular equivalence of Morita type with level.
\end{theorem}
\begin{proof}
Set $Y=\RHom _A(X, A)$. Since $_AX$ is compact, we have an isomorphism of functors
$$\RHom _A(X, -)\cong Y\otimes_A^L-:
\mathcal{D}(A)\rightarrow \mathcal{D}(B).$$ Hence, there is an adjoint pair
$$\xymatrix{ \mathcal{D}(B)\ar@<+1ex>[rr]^{X \otimes _B^L-}
&&\mathcal{D}(A)\ar@<+1ex>[ll]^{Y \otimes _A^L-}}$$
with unit $\eta$ and counit $\epsilon$.
Since $_BY\in K^b(\proj B)$, these adjoint functors restrict to one at the level of singularity categories
(ref. \cite[Lemma 1.2]{Orl04}).
By assumption, $X \otimes _B^L- : D_{sg}(B)\rightarrow D_{sg}(A)$ is an equivalence,
and it follows from \cite[Theorem 3.6]{Dal20} that there are two isomorphisms
$B\cong Y\otimes_A^L X$ in $D_{sg}(B^e)$ and
$A\cong X\otimes_B^L Y$ in $D_{sg}(A^e)$ (for this we need the assumption on
separability). Therefore, the mapping cones of $\eta _B : B\rightarrow Y\otimes_A^L X$
and $\epsilon _A : X\otimes_B^L Y\rightarrow A$ are perfect complexes of bimodules,
and by \cite[Proposition 3.8]{Wei17}, there exists some $l\in \mathbb{Z}$ such that
for any $i\geq l$, there are two isomorphisms
$\Omega ^i _{\mathcal{D}^b(B^e)}(B)\cong
\Omega ^i _{\mathcal{D}^b(B^e)}(Y\otimes_A^L X)$ and $\Omega ^i _{\mathcal{D}^b(A^e)}(A)\cong
\Omega ^i _{\mathcal{D}^b(A^e)}(X\otimes_B^L Y)$, up to some direct summands of projective bimodules
(see \cite[Proposition 3.5]{Wei17}).

From \cite[Proposition 4.4]{Dal20}, $X$ is isomorphic
in $\mathcal{D}(A\otimes _k B^{op})$ to a complex
$$0 \longrightarrow U \longrightarrow P_n \longrightarrow
\cdots \longrightarrow P_m \longrightarrow 0 ,$$
where all $P_i$ are finitely generated projective
$A$-$B$-bimodules and $U$ is finitely
generated projective as a left $A$-module and as a right $B$-module.
Similarly, $Y$ is isomorphic
in $\mathcal{D}(B\otimes _k A^{op})$ to a complex
$$0 \longrightarrow V \longrightarrow Q_{n'} \longrightarrow
\cdots \longrightarrow Q_{m'} \longrightarrow 0 ,$$
where all $Q_i$ are finitely generated projective
$B$-$A$-bimodules and $V$ is finitely
generated projective as a left $B$-module and as a right $A$-module.
Therefore, $Y\otimes_A^L X$ is quasi-isomorphic to the tensor product complex:
$$0 \longrightarrow V\otimes _AU \longrightarrow Z_{n+n'+1} \longrightarrow
\cdots \longrightarrow Z_{m+m'} \longrightarrow 0 ,$$
where all $Z_i$ are projective over $B^e$.
Hence, for any $i\geq n+n'+2$, we have $$\Omega ^i _{\mathcal{D}^b(B^e)}(Y\otimes_A^L X)\cong
\Omega ^{i-n-n'-2} _{B^e}(V\otimes _AU)
\cong
\Omega ^{i-n-n'-2} _{B-A}(V)\otimes _AU,$$
where the last isomorphism follows from the fact that
$U$ is projective as a left $A$-module and as a right $B$-module.
Similarly,  we obtain that
$$\Omega ^i _{\mathcal{D}^b(A^e)}(X\otimes_B^L Y)\cong
\Omega ^{i-n-n'-2} _{A^e}(U\otimes _BV)
\cong
U\otimes _B\Omega ^{i-n-n'-2} _{B-A}(V),$$
for any $i\geq n+n'+2$.
Taking $r=\max \{l, n+n'+2 \}$, we have isomorphisms
$$\Omega ^r_{B^e}(B)\cong \Omega ^r _{\mathcal{D}^b(B^e)}(Y\otimes_A^L X)\cong
\Omega ^{r-n-n'-2} _{B-A}(V)\otimes _AU$$ and
$$\Omega ^r_{A^e}(A)\cong \Omega ^r _{\mathcal{D}^b(A^e)}(X\otimes_B^L Y)\cong
U\otimes _B\Omega ^{r-n-n'-2} _{B-A}(V),$$
up to some projective direct summands.

Since $V$ is finitely
generated projective as a left $B$-module and as a right $A$-module,
$\Omega ^{r-n-n'-2} _{B-A}(V)$ is also finitely
generated projective as an one-side module.
Above all, we conclude that $(U, \Omega ^{r-n-n'-2} _{B-A}(V)) $
defines a singular equivalence of Morita type with level $r$.

\end{proof}

\begin{proposition}\label{prop-recoll}
Let $A$, $B$ and $C$ be finite dimensional $k$-algebras such that either
$A/\rad(A)$ or $B/\rad(B)$ is separable over $k$. If
$\gl B<\infty$ and $\mathcal{D}A$ admits a $2$-recollement relative to $\mathcal{D}B$
and $\mathcal{D}C$, then $A$ and $C$ are singularly equivalent of Morita type with level.
\end{proposition}
\begin{proof}
By \cite[Proposition 1]{QH16}, there is a
standard $2$-recollement
$$\xymatrix@!=4pc{ \mathcal{D}B \ar@<+1ex>[r]|{i_*} \ar@<-3ex>[r] & \mathcal{D}A
\ar@<+1ex>[r]|{j^*}\ar@<-3ex>[r] \ar@<-3ex>[l] \ar@<+1ex>[l]|{i^!} &
\mathcal{D}C \ar@<-3ex>[l] \ar@<+1ex>[l]|{j_*}}.$$
Then, it follows from \cite[Lemma 2.9]{AKLY17} that $i^!$ restricts to $\mathcal{D}^b(\mod )$,
and then $i^!(A)\in \mathcal{D}^b(\mod B)\cong K^b(\proj B)$ since $\gl B<\infty$. Therefore,
this $2$-recollement can be extended one step downwards, see \cite[Proposition 3.2]{AKLY17}.
So, there are four bimodule complexes $_CX_A$, $_AY_C$, $_AU_B$ and $_BV_A$
(which are perfect complexes of one-side modules)
such that $i_*=U\otimes _B^L-$, $i^!=V\otimes _A^L-$,
$j^*=X\otimes _A^L-$ and $j_*=Y\otimes _C^L-$, where $Y=\RHom _C(X,C)$.
On the other hand, it follows from \cite[Proposition 3]{Qin20} that
$j^*$ and $j_*$ induce a mutually inverse equivalence
between $D_{sg}(A)$ and $D_{sg}(C)$.

Since $j_*$ is fully faithful, there is an isomorphism
$C\cong X\otimes_A^LY$ in $\mathcal{D}C$. Moreover,
the canonical map $$X\otimes_A^LY=X\otimes_A^L\RHom _C(X,C)\rightarrow C,
x\otimes f\mapsto f(x)$$ is a morphism of $C$-$C$-bimodules.
Therefore, we get
$C\cong X\otimes_A^LY$ in $\mathcal{D}(C^e)$ and thus $C\cong X\otimes_A^LY$ in $D_{sg}(C^e)$.
Now we claim $A\cong Y\otimes_C^LX$ in $D_{sg}(A^e)$ and then
we are done by the proof of Theorem~\ref{theorem-sing}.

Let $\eta: 1_{\mathcal{D}(A)}\rightarrow Y\otimes_C^LX\otimes _A^L-$ be the unit of the adjoint pair
$$\xymatrix{ \mathcal{D}(A)\ar@<+1ex>[rr]^{X \otimes _A^L-}
&&\mathcal{D}(C)\ar@<+1ex>[ll]^{Y \otimes _C^L-}}.$$
Since these functors induce a singular equivalence, we have that
$\eta _A \otimes _A Z$ is an isomorphism in $D_{sg}(A)$,
for any $Z\in \mathcal{D}^b(A)$. If $A/\rad(A)$ is separable over $k$, then
it follows from \cite[Lemma 3.5]{Dal20} that
the mapping cone of $\eta _A$ is a perfect complex of $A$-$A$-bimodules,
and thus $A\cong Y\otimes_C^LX$ in $D_{sg}(A^e)$.

If $B/\rad(B)$ is separable over $k$, then the condition $\gl B<\infty$
implies that $B\in K^b(\proj B^e)$, see \cite[Lemma 7.2]{Rou08}. Hence $U\otimes _B^LV\cong U\otimes _B^LB\otimes_B^LV
\in K^b(\proj A^e)$, because the functors $-\otimes_B^LV: \mathcal{D}(B^e)\rightarrow \mathcal{D}(B\otimes _k A^{op})$
and $U\otimes _B^L- :\mathcal{D}(B\otimes _k A^{op})\rightarrow \mathcal{D}(A^e)$
restrict to $K^b(\proj )$. On the other hand, it follows from \cite[Theorem 1]{Han14} that
there is a recollement
$$\xymatrix @R=0.6in @C=0.8in{
\mathcal{D}(B\otimes _kA^{op})
\ar[r]|{U\otimes _B^L} & \mathcal{D}(A^e) \ar@<+3ex>[l]|{V\otimes _A^L}
\ar@<-3ex>[l]\ar[r]|{X\otimes _A^L-} & \mathcal{D}(C\otimes _kA^{op})
\ar@<+3ex>[l]|{Y\otimes _C^L-} \ar@<-3ex>[l] }.$$
Consequently, we have a triangle $U\otimes _BV\rightarrow A \rightarrow Y\otimes_C^LX\rightarrow$ in $\mathcal{D}(A^e)$.
Using the fact $U\otimes _BV\in K^b(\proj A^e)$, we obtain that $A\cong Y\otimes_C^LX$ in $D_{sg}(A^e)$.

\end{proof}

\begin{corollary}\label{cor-recoll}
Let $A$, $B$ and $C$ be finite dimensional $k$-algebras such that
$B$ has finite projective dimension as a $B$-$B$-bimodule. Assume $\mathcal{D}A$ admit a $2$-recollement relative to $\mathcal{D}B$
and $\mathcal{D}C$. Then $A$ and $C$ are singularly equivalent of Morita type with level.
\end{corollary}
\begin{proof}
From \cite[Chapter IX, Propositions 7.6]{CE56}, we have that $\gl B<\infty$.
Now note that $B\in K^b(\proj B^e)$, the statement can be proved in the same way as Proposition~\ref{prop-recoll}.
\end{proof}

\begin{corollary}{\rm (Compare \cite[Section 3]{Wang15})}\label{cor-tri-alg}
Let $A=\left(
\begin{array}{cccc}
B& 0  \\
_CM_B& C \\
\end{array}
\right)$, where $B, C$ are finite dimensional $k$-algebras and $M$ a finitely generated $C$-$B$-bimodules.
Then following statements hold:

{\rm (1)} If $B$ has finite projective dimension as a $B$-$B$-bimodule,
then $A$ and $C$ are singularly equivalent of Morita type with level;

{\rm (2)} If $C$ has finite projective dimension as a $C$-$C$-bimodule,
then $A$ and $B$ are singularly equivalent of Morita type with level.

\end{corollary}
\begin{proof}
(1) By \cite[Example 3.4]{AKLY17}, $\mathcal{D}A$ admits a $2$-recollement relative to $\mathcal{D}C$
and $\mathcal{D}B$. Moreover, $B\in K^b(\proj B^e)$ implies that $\gl B<\infty$,
and then $\pd M_B<\infty$. Therefore, this $2$-recollement can be extended one step
upwards, that is, $\mathcal{D}A$ admits a $2$-recollement relative to $\mathcal{D}B$
and $\mathcal{D}C$. Now the statement follows from Corollary~\ref{cor-recoll}.

(2) It follows from Corollary~\ref{cor-recoll}.

\end{proof}

Recall that an algebra $A$ is said to {\it satisfy the Fg condition} if the
Hochschild cohomology ring $HH^*(A)$ is a Noetherian ring,
and the Yoneda algebra
$\Ext _A^*(A/ \rad A, A/ \rad A)$ is a finitely generated
$HH^*(A)$-module. The following corollary generalizes \cite[Theorem 5]{Qin20}
on the assumption of separability.
\begin{corollary}{\rm (Compare \cite[Theorem 5]{Qin20})}\label{cor-fg}
Let $A$, $B$ and $C$ be finite dimensional $k$-algebras such that either
$A/\rad(A)$ or $B/\rad(B)$ is separable over $k$. If
$(\mathcal{D}B, \mathcal{D}A, \mathcal{D}C, i^*, i_*, i^!, j_!, j^*, j_*)$
is a recollement
and $j^*$ is an eventually homological isomorphism,
then $A$ satisfies the Fg condition if and only if so does $C$.
\end{corollary}
\begin{proof}
Since $j^*$ is an eventually homological isomorphism,
it follows from \cite[Theorem 1]{Qin20} that $\gl B<\infty$, and
the functor $i_*$ restricts to both $K^b(\proj)$ and $K^b(\inj)$.
By \cite[Proposition 3.2]{AKLY17}, this recollement can be extended one downwards,
and by Proposition~\ref{prop-recoll}, $A$ and $C$ are singularly equivalent of Morita type with level.
Now assume either $A$ or $C$ satisfies Fg. Then, it follows from
\cite[Theorem 1.5 (a)]{EHS04} that either $A$ or $C$ is Gorenstein, and by \cite[Theorem 3]{Qin20},
both $A$ and $C$ are Gorenstein. Now we finish our proof by \cite[Theorem 7.4]{Ska16}.
\end{proof}

In \cite{Chen14}, the author proved that a certain homological epimorphism between two
algebras induces a triangle equivalence between their singularity categories.
Now we will show that this equivalence is a
singular equivalence of Morita type with level.

\begin{theorem}{\rm (Compare \cite[Theorem]{Chen14})}\label{thm-homo-epic}
Let $A$ be a finite dimensional $k$-algebra and let $J\subseteq A$ be a homological
ideal which has finite projective dimension as an $A$-$A$-bimodule. Then
$A$ and $A/J$ are
singularly equivalent of Morita type with level.
\end{theorem}

\begin{proof}
consider the adjoint pair $\xymatrix{ \mathcal{D}(A)\ar@<+1ex>[rr]^{A/J \otimes _A^L-}
&&\mathcal{D}(A/J)\ar@<+1ex>[ll]^{A/J \otimes _{A/J} ^L-}}$ with the unit $\eta$ and the counit $\epsilon$.
Since $J$, as an $A$-$A$-bimodule, has finite
projective dimension, so it has finite projective dimension both as a left and right $A$-module.
Therefore, $\pd _A(A/J)<\infty$ and $\pd (A/J)_A<\infty$, and these adjoint functors restrict to
one at the level of singularity categories. By \cite[Theorem 4.4 (1)]{GL91},
the counit $\epsilon _{A/J}:A/J \otimes _A^L A/J\rightarrow A/J$ is an isomorphism in
$\mathcal{D}((A/J)^e)$ and thus $A/J\cong A/J \otimes _A^L A/J$ in $D_{sg}((A/J)^e)$.
Moreover, the unit map $\eta _{A}:A\rightarrow  A/J$ is an isomorphism in
$D_{sg}(A^e)$ since $\pd _{A^e}(J)<\infty$.
As a consequence, we obtain that $A$ and $A/J$
are singularly equivalent of Morita type with level by the same way
as we did in Theorem~\ref{theorem-sing}.
\end{proof}

\section{Singular equivalences induced by idempotents}
\indent\indent Let $A$ be a finite-dimensional algebra over a field $k$, and let $e \in A$ be an idempotent.
From \cite{PSS14}, the functor $eA\otimes _A-: \mod A\rightarrow
\mod eAe$ induces a singular equivalence between
$A$ and $eAe$ if and only if $\pd _A(\frac {A/AeA}{\rad (A/AeA)})<\infty$ and $\pd _{eAe}eA<\infty$,
see also \cite[Theorem 2.1]{Chen09}. Similarly,  $Ae\otimes _{eAe}-: \mod eAe\rightarrow
\mod A$ induces a singular equivalence if and only if $\id _A(\frac {A/AeA}{\rad (A/AeA)})<\infty$ and $\pd Ae_{eAe}<\infty$,
see \cite[Theorem II]{Shen20}. In this section, we will show that $A$ and $eAe$ are
singularly equivalent of Morita type with level in these two cases .

\begin{theorem}\label{thm-idem-pd}
Let $A$ be a finite-dimensional algebra over a field $k$ such that
$A/\rad(A)$ is separable over $k$ and let $e \in A$ be an idempotent.
If $\pd _A(\frac {A/AeA}{\rad (A/AeA)})
<\infty$ and $\pd _{eAe}eA<\infty$, then
$A$ and $eAe$ are
singularly equivalent of Morita type with level.
\end{theorem}
\begin{proof}
Consider the recollement
$$\xymatrix@!=6pc{ \mathcal{D}B \ar[r]|{i_*} & \mathcal{D}A \ar@<-3ex>[l]|{i^*}
\ar@<+3ex>[l]|{i^!} \ar[r]|{j^*=eA\otimes_A^L-} & \mathcal{D}(eAe)
\ar@<-3ex>[l]|{j_!=Ae\otimes _{eAe}^L} \ar@<+3ex>[l]|{j_*}},$$ where $B$ is a dg algebra.
Since $\pd _{eAe}eA<\infty$, it follows from \cite[Lemma 2.5]{AKLY17}
that $j^*$ restricts to $K^b(\proj)$, and then
$j_*$ restricts to $\mathcal{D}^b(\mod )$ by \cite[Lemma 2.7]{AKLY17}.
Moreover, it follows from \cite[Lemma 2.9 (e)]{AKLY17} that $j^*$ restricts to $\mathcal{D}^b(\mod )$.
Now we claim that $j_*$ restricts to $K^b(\proj)$,
and then $j^*$ and $j_*$ induce an adjoint pair between the corresponding singularity categories.

Since $j^*$ and $j_*$ restrict to $\mathcal{D}^b(\mod )$, we infer that
$j_*j^*A\in \mathcal{D}^b(\mod A)$, and then $i_*i^!A\in \mathcal{D}^b(\mod A)$ by the triangle
$$i_*i^!A\rightarrow A\rightarrow j_*j^*A\rightarrow.$$
As $i_*i^!A\in \Ker j^*$, it follows that all homologies
of $i_*i^!A$ are in $\mod A/AeA$  (ref. \cite[Proposition
2.17]{Miy03}). Thus $i_*i^!A\in \tria (\mod A/AeA)=\tria (\frac {A/AeA}{\rad (A/AeA)})$,
and then the assumption $\pd _A(\frac {A/AeA}{\rad (A/AeA)})<\infty$ implies that
$i_*i^!A\in K^b(\proj A)$. Applying the triangle given above, we conclude that
$j_*j^*A\in K^b(\proj A)$. For any $X\in \mathcal{D}(eAe)$,
we have $X\cong j^*j_*X\subseteq j^*(\Tria A)\subseteq \Tria j^*A$, that is,
$j^*A$ is a compact generator of $\mathcal{D}(eAe)$. So, it follows that $\thick j^*A=\thick (eAe)$,
and thus $\thick j_*j^*A=\thick j_*(eAe)$ for $j_*$ is a full embedding, cf. \cite[ Lemma 2.2]{Nee92}.
Now note that $j_*j^*A\in K^b(\proj A)$, and so we get $j_*(eAe)\in K^b(\proj A)$, that is,
the functor $j_*$ restricts to $K^b(\proj)$.

Next, we claim that ($j^*$, $j_*$) induces a mutually inverse equivalence
between the corresponding singularity categories, and then
we obtain that $A$ and $eAe$ are
singularly equivalent of Morita type with level
by the same way as we did in
Theorem~\ref{theorem-sing}.
For this, take $X\in \mathcal{D}^b(\mod A)$ and consider the triangle
$$i_*i^!X\rightarrow X\rightarrow j_*j^*X\rightarrow.$$
By a similar argument as above, we can prove that $i_*i^!X \in K^b(\proj A)$.
Therefore,
$X\cong j_*j^*X$ in $D_{sg}(A)$, and for any $Y\in \mathcal{D}^b(\mod eAe)$ the isomorphism
$Y\cong j^*j_*Y$ in $D_{sg}(eAe)$ is clear.
Hence, ($j^*$, $j_*$) induces a mutually inverse equivalence
between the corresponding singularity categories.

\end{proof}

\begin{theorem}\label{thm-idem-id}
Let $A$ be a finite-dimensional algebra over a field $k$ such that
$A/\rad(A)$ is separable over $k$ and let $e \in A$ be an idempotent.
If $\id _A(\frac {A/AeA}{\rad (A/AeA)})<\infty$ and $\pd Ae _{eAe}<\infty$, then
$A$ and $eAe$ are
singularly equivalent of Morita type with level.
\end{theorem}
\begin{proof}
Consider the recollement
$$\xymatrix@!=6pc{ \mathcal{D}B \ar[r]|{i_*} & \mathcal{D}A \ar@<-3ex>[l]|{i^*}
\ar@<+3ex>[l]|{i^!} \ar[r]|{j^*=eA\otimes_A^L-} & \mathcal{D}(eAe)
\ar@<-3ex>[l]|{j_!=Ae\otimes _{eAe}^L} \ar@<+3ex>[l]|{j_*}},$$ where $B$ is a dg algebra.
Since $\pd Ae _{eAe}<\infty$, it follows from \cite[Lemma 2.8]{AKLY17}
that $j_!$ has a left adjoint $j^\theta =\RHom _{eAe}(Ae, eAe)\otimes _A^L-$.
By \cite[Lemma 2.2]{AKLY17},
$i^*$ also has a left adjoint $i_\theta$. Therefore, we have a
$2$-recollememt
$$\xymatrix@!=4pc{ \mathcal{D}(eAe) \ar@<+1ex>[r]|{j_!} \ar@<-3ex>[r]|{j_*}  & \mathcal{D}A
\ar@<+1ex>[r]|{i^*}\ar@<-3ex>[r]|{i^!} \ar@<-3ex>[l]|{j^\theta} \ar@<+1ex>[l]|{j^*} &
\mathcal{D}B \ar@<-3ex>[l]|{i_\theta} \ar@<+1ex>[l]|{i_*}}.$$
It follows from \cite[Lemma 2.9 (e)]{AKLY17} that
$j^\theta$ restricts to $K^b(\proj)$, and
$j_!$ restricts to $\mathcal{D}^b(\mod)$ and $K^b(\proj)$.
Now we claim that $j^\theta$ restricts to $\mathcal{D}^b(\mod)$,
and then ($j^\theta$, $j_!$) induces an adjoint pair between the corresponding singularity categories.

Since $j^*$ and $j_!$ restrict to $\mathcal{D}^b(\mod )$, we infer that
$j_!j^*DA\in \mathcal{D}^b(\mod A)$, and then $i_*i^*DA\in \mathcal{D}^b(\mod A)$ by the triangle
$$j_!j^*DA\rightarrow DA\rightarrow i_*i^*DA\rightarrow.$$
As $i_*i^*DA\in \Ker j^*$, it follows that all homologies
of $i_*i^*DA$ are in $\mod A/AeA$. Therefore, $i_*i^*DA\in \tria (\mod A/AeA)=\tria (\frac {A/AeA}{\rad (A/AeA)})$,
and the assumption $\id _A(\frac {A/AeA}{\rad (A/AeA)})<\infty$ implies that
$i_*i^*DA\in K^b(\inj A)$. Applying the triangle given above, we conclude that
$j_!j^*DA\in K^b(\inj A)$. Therefore, $$j_!D(eAe)\cong j_!j^*j_*D(eAe)\subseteq
j_!j^*(\thick DA)\subseteq \thick j_!j^*DA \subseteq K^b(\inj A).$$
Here, the first inclusion follows from the fact that $j_*$ restrict to
$K^b(\inj A)$, see \cite[Lemma 1]{QH16}. As a result,
$j_!$ restricts to $K^b(\inj )$ and then $j^\theta$ restricts to $\mathcal{D}^b(\mod)$ by \cite[Lemma 1]{QH16}.

Next, we claim that ($j^\theta$, $j_!$) induces a mutually inverse equivalences
between the corresponding singularity categories, and then
we obtain that $A$ and $eAe$ are
singularly equivalent of Morita type with level
by the same way as we did in
Theorem~\ref{theorem-sing}.
For this, take $X\in \mathcal{D}^b(\mod A)$ and consider the triangle
$$i_\theta i^*X\rightarrow X\rightarrow j_!j^\theta X\rightarrow.$$
Now we will use \cite[Lemma 2.4 (c)]{AKLY17} to show that $i_\theta i^*X\in K^b(\proj A)$.
For any $Y\in \mathcal{D}^b(\mod A)$ and $n\in \mathbb{Z}$, we have isomorphisms
$$\Hom _{\mathcal{D}A}(i_\theta i^*X, Y[n])\cong \Hom _{\mathcal{D}B}(i^*X, i^*Y[n])
\cong \Hom _{\mathcal{D}A}(X, i_*i^*Y[n]).$$
By a similar argument as above, we can prove that $i_*i^*Y \in K^b(\inj A)$,
and then $\Hom _{\mathcal{D}A}(X, i_*i^*Y[n])=0$ for all but finitely many $n$.
As a result, $\Hom _{\mathcal{D}A}(i_\theta i^*X, Y[n])=0$ for all but finitely many $n$,
and then $i_\theta i^*X\in K^b(\proj A)$ by \cite[Lemma 2.4 (c)]{AKLY17}. Applying the triangle given above,
we infer that $X\cong j_!j^\theta X$ in $D_{sg}(A)$. For any $Z\in \mathcal{D}^b(\mod eAe)$, the isomorphism
$Z\cong j^\theta j_!Z$ in $D_{sg}(eAe)$ is clear.
Therefore, ($j^\theta$, $j_!$) induces a mutually inverse equivalences
between the corresponding singularity categories.
\end{proof}

Next, we will show that the condition $\pd _{eAe}eA<\infty$ (resp. $\pd Ae _{eAe}<\infty$)
in Theorem~\ref{thm-idem-pd} (resp. Theorem~\ref{thm-idem-id}) can be removed if
$\pd _A(\frac {A/AeA}{\rad (A/AeA)})\leq 1$ (resp. $\id _A(\frac {A/AeA}{\rad (A/AeA)})\leq 1$).
\begin{proposition}\label{prop-idem}
Let $A$ be a finite-dimensional algebra over a field $k$ such that
$A/\rad(A)$ is separable over $k$ and let $e \in A$ be an idempotent.
If $\pd _A(\frac {A/AeA}{\rad (A/AeA)})\leq 1$ or $\id _A(\frac {A/AeA}{\rad (A/AeA)})\leq 1$, then
$A$ and $eAe$ are
singularly equivalent of Morita type with level.
\end{proposition}
\begin{proof}
If $\pd _A(\frac {A/AeA}{\rad (A/AeA)})\leq 1$ or
$\id _A(\frac {A/AeA}{\rad (A/AeA)})\leq 1$, then the restriction functor
$\mod A/AeA \rightarrow \mod A$
is a homological embedding (ref. \cite[Proposition 3.5 (iv), Remark 5.9]{GPS18}).
Therefore, there is a recollement of derived categories
$$\xymatrix@!=6pc{ \mathcal{D}(A/AeA) \ar[r]|{i_*} & \mathcal{D}A \ar@<-3ex>[l]|{i^*}
\ar@<+3ex>[l]|{i^!} \ar[r]|{j^*=eA\otimes_A^L-} & \mathcal{D}(eAe)
\ar@<-3ex>[l]|{j_!=Ae\otimes _{eAe}^L} \ar@<+3ex>[l]|{j_*}},$$
where $i_*$
sends all $A/AeA$-modules to $A$-modules. If $\pd _A(\frac {A/AeA}{\rad (A/AeA)})\leq 1$,
then $\pd _A(i_*M)\leq 1$ for any $M\in \mod A/AeA$. Therefore, the functor $i_*$ restricts to
$K^b(\proj )$, and so does $j^*$ (see \cite[Lemma 2.5 and Lemma 4.3]{AKLY17}.
Hence, $j^*A=eA \in K^b(\proj eAe)$ and then $\pd _{eAe}eA<\infty$.
Therefore, the desired result follows from
Theorem~\ref{thm-idem-pd}, and
the case of $\id _A(\frac {A/AeA}{\rad (A/AeA)})\leq 1$ can be proved dually.

\end{proof}

The following example illustrates that our result
can be applied to construct singular equivalence of Morita type with level
where previous method seems too complicated.
\begin{example}{\rm
Let $A$ be the triangular matrix algebra
$\left( \begin{array}{cccc}
k& k  \\
0& k[x]/\langle x^2\rangle \\
\end{array}
\right)$,
where $k$ is an algebraically closed field.
In \cite[Example 7.5]{Ska16}, the author
proved that $A$ and $k[x]/\langle x^2\rangle$
are singularly equivalent of Morita type with level by constructing two explicit
bimodules. Now we will use Theorem~\ref{thm-idem-pd} to give
a brief proof.
Let $e=\left(
\begin{array}{cccc}
0& 0  \\
0& 1 \\
\end{array}
\right)$
be an idempotent. Then it is shown in \cite[Example 5.5]{PSS14}
that $\pd _A(\frac {A/AeA}{\rad (A/AeA)})
<\infty$ and $\pd _{eAe}eA<\infty$. Therefore, it follows from Theorem~\ref{thm-idem-pd} that
$A$ and $eAe=k[x]/\langle x^2\rangle$ are
singularly equivalent of Morita type with level.

}
\end{example}

\section{Syzygy-finite, injectives generation and singular equivalences}
\indent\indent In \cite[Lemma 4.13]{Wang15}, Wang proved that the finiteness of the finitistic dimension
is invariant under singular equivalences of Morita type with level. In this section,
we will show that the properties of syzygy-finite, Igusa-Todorov, injectives generation
and projectives cogeneration
are also preserved under singular equivalences of Morita type with level.
\begin{proposition}\label{prop-syz}
Let $A$ and $B$ be two finite dimensional $k$-algebras
which are singularly equivalent of Morita type with level. Then $A$ is syzygy-finite
(resp. Igusa-Todorov) if and only if $B$ is syzygy-finite
(resp. Igusa-Todorov).
\end{proposition}

\begin{proof}
Assume that $(_AM_B, _BN_A)$ defines a singular equivalence of Morita type of level $n$ between $A$ and $B$.
Consider the functors $F=N\otimes _A-: \mathcal{D}A\rightarrow \mathcal{D}B$ and
$G=M\otimes _B-: \mathcal{D}B\rightarrow \mathcal{D}A$. Since $M$ and $N$ are projective as one-side modules,
it follows that the functors $F$ and $G$ restrict to $\mathcal{D}^b(\mod )$ and $K^b(\proj )$.
Assume that $B$ is syzygy-finite
(resp. Igusa-Todorov). Then $F(\mod A)$ is syzygy-finite
(resp. Igusa-Todorov) since $F(\mod A)\subseteq \mod B$,
and by \cite[Lemma 3.3]{WW20}, $GF(\mod A)=\Omega ^n(\mod A)$ is also syzygy-finite
(resp. Igusa-Todorov). Hence, $A$ is syzygy-finite
(resp. Igusa-Todorov) and the ``only if'' part can be proved parallelly.
\end{proof}

\begin{proposition}\label{prop-inj-gen}
Let $A$ and $B$ be two finite dimensional $k$-algebras
which are singularly equivalent of Morita type with level. Then injectives
generate (resp. projectives cogenerate) for $A$ if and only if injectives generate
(resp. projectives cogenerate) for $B$.
\end{proposition}
\begin{proof}
Assume that $(_AM_B, _BN_A)$ defines a singular equivalence of Morita type of level $n$ between $A$ and $B$.
Then the functors $F=\Hom _A(M,-): \mathcal{D}A\rightarrow \mathcal{D}B$ and
$G=\Hom _B(N,-): \mathcal{D}B\rightarrow \mathcal{D}A$ preserve bounded complexes of
injectives and set indexed coproducts.
Suppose that injectives generate for $B$. Then $\mathcal{D}B= \Tria DB$, and thus
$GF(A)\in G(\mathcal{D}B)= G(\Tria DB)\subseteq \Tria G(DB)\subseteq \Tria DA$.
On the other hand, $GF(A)=\Hom _B(N,\Hom _A(M,A))\cong \Hom _A(M\otimes _BN,A)
\cong \Hom _A(\Omega ^n_{A^e}(A),A)$, which is isomorphic to $\Hom _A(A, \Omega ^{-n}_{A}(A))$
in $\overline{\mod }A$, see \cite[Lemma 4.14]{Wang15}. Therefore, we conclude that
$GF(A)\cong \Omega ^{-n}_{A}(A)$, up to some direct summands of injective $A$-modules.
Since $GF(A)\in \Tria DA$, we get $\Omega ^{-n}_{A}(A)\in \Tria DA$
and then $A \in \Tria DA$. Thus, $\mathcal{D}A=\Tria A\subseteq \Tria DA$,
and then $\mathcal{D}A=\Tria DA$, that is, injectives generate for $A$.
The ``only if'' part can be proved parallelly.

Clearly, both $N\otimes _A-:\mathcal{D}A\rightarrow \mathcal{D}B$ and
$M\otimes _B-:\mathcal{D}B\rightarrow \mathcal{D}A$ preserve bounded complexes of projective modules.
Further, by \cite[Lemma 2.8]{AKLY17}, these two functors
have left adjoints, and then
they preserve set indexed products. Therefore,
the statement on ``projectives cogenerate'' can be proved similarly.
\end{proof}

The following corollary follows from Theorem~\ref{thm-homo-epic},
Proposition~\ref{prop-syz} and Proposition~\ref{prop-inj-gen}.
\begin{corollary}\label{coro-homo-epic-syz}
Let $A$ be a finite dimensional $k$-algebra and let $J\subseteq A$ be a homological
ideal which has finite projective dimension as an $A$-$A$-bimodule. Then
$A$ has the property of syzygy-finite (resp. Igusa-Todorov, injectives
generation, projectives cogeneration) if and only if so does $A/J$.
\end{corollary}

In \cite[Proposition 7.6]{Cum19}, the author investigated the invariance of
injectives generation (resp. projectives cogeneration) under the operation
of vertex removal. Now we mention that \cite[Proposition 7.6]{Cum19} can be completed
as follow if we only consider algebras with the conditions of separability.

\begin{corollary}{\rm (Compare \cite[Proposition 7.6]{Cum19})}\label{cor-iedm-inj-gen}
Let $A$ be a finite-dimensional $k$-algebra such that
$A/\rad(A)$ is separable over $k$. Let $e \in A$ be an idempotent and assume
$Ae\otimes _{eAe}^LeA$ is bounded in
cohomology. If $\pd _A(\frac {A/AeA}{\rad (A/AeA)})
<\infty$ or $\id _A(\frac {A/AeA}{\rad (A/AeA)})
<\infty$, then $A$ has the property of syzygy-finite (resp. Igusa-Todorov, injectives
generation, projectives cogeneration) if and only if so does $eAe$.
\end{corollary}

\begin{proof}
Consider the functors $i^*, i_*, i^!, j_!,j^*$ and $j_*$ in the proof of Theorem~\ref{thm-idem-pd}.
If $Ae\otimes _{eAe}^LeA\in \mathcal{D}^b(A)$ and $\pd _A(\frac {A/AeA}{\rad (A/AeA)})
<\infty$, then $j_!j^*A\in \mathcal{D}^b(A)$
and thus $i_*i^*A \in \mathcal{D}^b(A)$ by the canonical triangle
$j_!j^*A\rightarrow A\rightarrow i_*i^*A\rightarrow.$
Since $i_*i^*A\in \Ker j^*$, it follows that all cohomologies
of $i_*i^*A$ are in $\mod A/AeA$. Therefore, $i_*i^*A\in \tria (\mod A/AeA)=\tria (\frac {A/AeA}{\rad (A/AeA)})
\in K^b(\proj A)$. Applying the triangle given above, we have
$j_!j^*A\in K^b(\proj A)$ and then $j^*A=eA\in K^b(\proj eAe)$, see
\cite[Lemma 4.2]{AKLY17}. Therefore, $\pd _{eAe}eA<\infty$ and by
Theorem~\ref{thm-idem-pd}, $A$ and $eAe$ are singularly equivalent of Morita type with level.
Now the statement follows from
Proposition~\ref{prop-syz} and Proposition~\ref{prop-inj-gen}.

If $Ae\otimes _{eAe}^LeA\in \mathcal{D}^b(A)$ and $\id _A(\frac {A/AeA}{\rad (A/AeA)})
<\infty$, then $j_!j^*A\in \mathcal{D}^b(A)$.
Therefore, we obtain $j_*j^*(DA)\in \mathcal{D}^b(A)$ by the isomorphisms
$$H^n(j_*j^*(DA))\cong \Hom _{\mathcal{D}A}(A, j_*j^*(DA)[n])\cong \Hom _{\mathcal{D}A}(j_!j^*A, DA[n]),$$
and then $i_*i^!(DA)\in \mathcal{D}^b(A)$ by the canonical triangle
$i_*i^!(DA)\rightarrow DA\rightarrow j_*j^*(DA)\rightarrow.$
Since $i_*i^!(DA)\in \Ker j^*$, it follows that all cohomologies
of $i_*i^!(DA)$ are in $\mod A/AeA$. Therefore, $i_*i^!(DA)\in \tria (\mod A/AeA)=\tria (\frac {A/AeA}{\rad (A/AeA)})
\in K^b(\inj A)$. Applying the triangle given above, we get
$j_*j^*(DA)\in K^b(\inj A)$. Now we claim $j_!$ restricts to $\mathcal{D}^b(\mod)$.
For this, take $X\in \mathcal{D}^b(eAe)$ and consider the isomorphisms
$$DH^n(j_!X)\cong H^{-n}(D(j_!X))\cong \Hom_{\mathcal{D}k}(j_!X,k[-n])
\cong \Hom_{\mathcal{D}A}(j_!X,DA[-n]),$$
where the last isomorphism follows by adjunction. Therefore,
\begin{equation}\label{iso}
DH^n(j_!X)\cong \Hom_{\mathcal{D}(eAe)}(X,j^*DA[-n])\cong \Hom_{\mathcal{D}A}(j_*X,j_*j^*DA[-n]).
\end{equation}
Since $H^n(j_*X)\cong \Hom_{\mathcal{D}A}(A,j_*X[n])\cong \Hom_{\mathcal{D}(eAe)}(j^*A,X[n])$,
we infer that $H^n(j_*X)=0$, for any sufficiently small $n$.
As $j_*j^*(DA)\in K^b(\inj A)$, it follows from the formula~\ref{iso}
that $H^n(j_!X)=0$ for sufficiently small $n$. Moreover, by
\cite[Lemma 2.10 (b)]{AKLY17}, we get $H^n(j_!X)=0$ for sufficiently large $n$.
Therefore, $j_!$ restricts to $\mathcal{D}^b(\mod)$ and thus
$\pd Ae_{eAe}<\infty$, see \cite[Lemma 2.8 ]{AKLY17}.
Now the statement follows from Theorem~\ref{thm-idem-id},
Proposition~\ref{prop-syz} and Proposition~\ref{prop-inj-gen}.
\end{proof}

Applying Proposition~\ref{prop-idem}, Proposition~\ref{prop-syz}
and Proposition~\ref{prop-inj-gen}, we get the following result
for path algebras.
\begin{corollary}\label{cor-path-alg}
Let $A=kQ/I$ be a quotient of a path algebra
such that
$A/\rad(A)$ is separable over $k$. Choose
a vertex $v$ in $Q$ where no relations start or no relations end, and let $e$ be the
sum of all idempotents corresponding to
vertices except
$v$. Then $A$ has the property of syzygy-finite (resp. Igusa-Todorov, injectives
generation, projectives cogeneration) if and only if so does $eAe$.
\end{corollary}
\begin{proof}
Clearly, $\frac {A/AeA}{\rad (A/AeA)}$ is nothing but the simple module
corresponding to the vertex $v$. Hence, it follows from
\cite[Corollary, Section 1.1]{Bon83} that $\pd _A(\frac {A/AeA}{\rad (A/AeA)})\leq 1$
(resp. $\id _A(\frac {A/AeA}{\rad (A/AeA)})\leq 1$)
if and only if no relation starts (resp. ends) in $v$.
Now, the statement follows from Proposition~\ref{prop-idem},
Proposition~\ref{prop-syz} and Proposition~\ref{prop-inj-gen}.

\end{proof}

Next, we will apply our results to produce some examples of
algebras which have the properties of syzygy-finite, Igusa-Todorov and injectives generation.
\begin{example}{\rm (\cite[Example 3.3]{Chen14})
Let $\Gamma$ be the $k$-algebra given by the following quiver
$$\xymatrix{1 \ar@<+0.5ex>[r]^\alpha & \ast
\ar@<+0.5ex>[l]^\beta \ar@(ul,ur)^x \ar@<+0.5ex>[r]^\delta & 2
\ar@<+0.5ex>[l]^\gamma}$$
with relations
$\{x^2, \delta x, \beta x, x\gamma, x \alpha, \beta \gamma, \delta \alpha,
\beta \alpha , \delta \gamma, \alpha \beta -\gamma \delta\}$. We write the concatenation of paths
from right to left.
The singularity category of $\Gamma$ is studied in \cite[Example 3.3]{Chen14}.
Indeed, there is an equivalence $D_{sg}(\Gamma)\cong D_{sg}(\Gamma /\Gamma e_1 \Gamma)$
induced by a homological ideal.
Note that $\Gamma /\Gamma e_1 \Gamma$ is a monomial algebra, and then
it has the properties of syzygy-finite and injectives generation.
Applying Corollary~\ref{coro-homo-epic-syz}, we
obtain that $\Gamma$ also has the properties of syzygy-finite and injectives generation.
}
\end{example}

\begin{example}
{\rm (\cite[Example 4.4]{FS92}) Let $A$ be the $k$-algebra given by the following quiver
$$
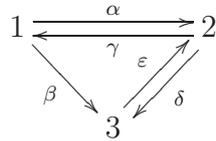
\xymatrix{1 \ar@<+0.5ex>[rr]^\alpha \ar[rd]_\beta & & 2 \ar@<+0.5ex>[ld]^\delta
\ar@<+0.5ex>[ll]^\gamma \\
& 3 \ar@<+0.5ex>[ru]^\varepsilon & }$$
with relations
$\{\delta \alpha, \gamma \alpha, \beta \gamma, \gamma \varepsilon \beta, \delta \varepsilon \delta,
   \alpha \gamma -\varepsilon \delta\}$.
Let $e=e_1+e_2$, and then $eAe$ is the monomial algebra
$\xymatrix{1 \ar@<+1.5ex>[r]^{\tau _{\alpha}} \ar@<-1.5ex>[r]_{\tau _{\varepsilon \beta}} & 2 \ar[l]|{\tau _{ \gamma}}
  }$ with relations $\tau _{ \gamma}\tau _{\alpha}, \tau _{\varepsilon \beta} \tau _{ \gamma},
  \tau _{ \gamma} \tau _{\varepsilon \beta}$. Therefore, the algebra
$eAe$ has the properties of syzygy-finite and injectives generation.
Clearly, there is no relation starts in $3$, and it follows from Corollary~\ref{cor-path-alg}
that $A$ also has the properties of syzygy-finite and injectives generation.
}
\end{example}

\noindent {\footnotesize {\bf ACKNOWLEDGMENT.}
The author is grateful to Ren Wang for many helpful
discussions and suggestions. This work is supported by the National Natural Science
Foundation of China (Grant No.12061060 and 11701321).}

\end{document}